\documentclass[11pt,reqno]{amsart}
\usepackage{amsfonts,amsmath,amssymb}
\usepackage{verbatim}
\newcommand{\what}{\widehat}%
\newcommand{\R}{\mathbb R}%
\newcommand{\C}{\mathbb C}%
\newcommand{\Z}{\mathbb Z}%
\newtheorem{theorem}{Theorem}[section]
\newtheorem{lemma}[theorem]{Lemma}

\newtheorem{corollary}[theorem]{Corollary}

\theoremstyle{definition}

\theoremstyle{definition}

\numberwithin{equation}{section}

\begin{document}
\title{Wiener Tauberian theorem for hypergeometric transforms}

\author{Sanjoy Pusti and Amit Samanta}

\address[Sanjoy Pusti]{Department of Mathematics and Statistics; Indian Institute of Technology, Kanpur-208016, India.}
\email{spusti@iitk.ac.in}

\address[Amit Samanta]{Department of Mathematics and Statistics; Indian Institute of Technology, Kanpur-208016, India.}
\email{amit.gablu@gmail.com}

\thanks{The second author is financially supported by NBHM, Govt. of India}
\thanks{We are thankful to Prof. Angela Pasquale for several helpful suggestions.}

\begin{abstract}
We prove a genuine analogue of Wiener Tauberian theorem for hypergeometric transforms. As an application we prove analogue of Furstenberg theorem on Harmonic functions.
\end{abstract}

\maketitle
\section{Introduction}
A famous theorem of Norbert Wiener states that for a function $f\in L^1(\R)$, span of translates $f(x-a)$ with complex coefficients is dense in $L^1(\R)$ if and only if the Fourier transform $\what{f}$ is nonvanishing on $\R$. This theorem has been extended to abelian groups. The hypothesis (in the abelian case) is on a Haar integrable function which has nonvanishing Fourier transform on all unitary characters. However, Ehrenpreis and Mautner (in \cite{EM1}) has observed that Wiener Tauberian theorem fails even for the commutative Banach algebra of integrable radial functions
on $\mathrm {SL}(2, \R)$. A modified version of the theorem was established in \cite[Theorem 6]{EM1} for radial functions in $L^1(\mathrm{SL}(2, \R))$. In their theorem they prove that if a function $f$ satisfies ``not-to-rapidly decay'' condition and nonvanishing condition on some {\em extended strip}, etc., then the ideal generated by $f$ is dense in $L^1(\mathrm{SL}(2, \R)//\mathrm{SO}(2))$. This has been  extended to all rank one semisimple Lie groups in the $K$-biinvariant setting (see \cite{Ben-0}, \cite{Sarkar-1998}) with the extended strip condition.  
The same theorem has been extended for hypergeometric transforms (in \cite{Liu}). Further references in this literature are  \cite{Sarkar-1997},  \cite{Sitaram-1980}, \cite{Sitaram-1988}, \cite{Naru-2009}, \cite{Naru-2011}, \cite{Pusti-2011}. In (\cite{Ben-1, Ben-2}), a genuine analogue of Wiener Tauberian theorem is proved for $\mathrm{SL}(2, \R)$ in the $K$-biinvariant setting without the extended strip condition. Following their method we have extended this result to all real rank one semisimple Lie groups in the $K$-biinvariant settings (\cite{PS}).   In this paper we prove Wiener Tauberian theorem for hypergeometric transforms in the exact strip.  Let $\alpha\geq\beta\geq-\frac{1}{2}$, $\alpha\neq-\frac{1}{2}$, $\rho=\alpha + \beta + 1$, $S_1=\{\lambda\in\C\mid |\Im\lambda|\leq \rho\}$ and
$$
\Delta_{\alpha, \beta}(t)=(2|\sinh t|)^{2\alpha + 1} (2\cosh t)^{2\beta + 1}, \,\,t\in\R.
$$
Let $L^1(\R, \Delta_{\alpha, \beta})_e$ be the collection of even functions $f$ such that $\|f\|_1:=\int_\R |f(t)| \Delta_{\alpha, \beta}(t)\,dt<\infty$. Also let $L^1_0(\R, \Delta_{\alpha, \beta})_e$ be the collection of  functions $f\in L^1(\R, \Delta_{\alpha, \beta})_e$  such that $$\int_\R f(t) \Delta_{\alpha, \beta}(t)\, dt=0.$$ For $f\in L^1(\R,\Delta_{\alpha,\beta})_e$, $\widehat{f}=\widehat{f}^{(\alpha,\beta)}$ denotes the Fourier-Jacobi transform of $f$ (see preliminaries for the definition). 

For any function $F$ on $S_1$, we define $$\delta_\infty^+(F)=-\limsup_{t\rightarrow\infty} e^{-\frac{\pi}{2\rho}t}\log|F(t)|, \,\,\,\,\,\textup{and} \,\,\,\,\,\delta_{ i\rho}(F)=\limsup_{t\rightarrow \rho-} (\rho-t)\log|F(i t)|.$$ 

Our theorem states that,

\begin{theorem}\label{main-theorem}
Let $\{f_\nu\mid \nu\in \Lambda\}$ be a collection of functions in $L^1(\R, \Delta_{\alpha, \beta})_e$ and $I$ be the smallest closed ideal in $L^1(\R, \Delta_{\alpha,\beta})_e$ containing $\{f_\nu\mid \nu\in \Lambda\}$. 
\begin{enumerate}
\item Suppose that element of $\{\what{f_\nu}\mid\nu\in \Lambda\}$ has no common zero in $S_1$ and $\inf_{\nu\in\Lambda}\delta_\infty^+(\what{f_\nu})=0$. Then $I=L^1(\R, \Delta_{\alpha, \beta})_e$.

\item Suppose that $\{\pm i\rho\}$ is the only common zero of $\{\what{f_\nu}\mid\nu\in \Lambda\}$ in $S_1$ and $\inf_{\nu\in\Lambda}\delta_\infty^+(\what{f_\nu})=\inf_{\nu\in\Lambda}\delta_{i\rho}(\what{f_\nu})=0$. Then $I=L^1_0(\R, \Delta_{\alpha, \beta})_e$.

\end{enumerate}
\end{theorem}

Most of the part of the proof of this theorem similar to our earlier paper (\cite{PS}). Theorefore we state such results without any proof. The proofs will follow as in \cite{PS}.

As an application of this theorem we prove  Frustenburg type theorem on Harmonic functions, following \cite{Ben-2}.

\section{Preliminaries}
Most of our notation related to the hypergeometric functions is standard and  can be found for example in
\cite{Koornwinder}. 
We shall follow the standard practice of using the letter $C$ for constants, whose value may change from one line to another. Everywhere in this article the symbol
$f_1\asymp f_2$ for two positive expressions $f_1$ and $f_2$ means that there are positive constants $C_1, C_2$ such
that $C_1f_1\leq f_2\leq C_2f_1$. For a complex number $z$, we will use $\Re z$ and $\Im z$ to
denote respectively the real and imaginary parts of $z$. 

A Jacobi function $\phi_\lambda^{(\alpha, \beta)} (\alpha, \beta, \lambda\in\C, \alpha\not=-1, -2, \cdots)$ is defined as the even $C^\infty$ function on $\R$ such that $\phi_\lambda^{(\alpha, \beta)}(0)=1$ and it satisfies the following differential equation 
\begin{equation}\label{eqn-1}
\left(\frac{d^2}{dt^2} + ((2\alpha +1)\coth t + (2\beta +1)\tanh t)\frac{d}{dt} + \lambda^2 + (\alpha +\beta +1)^2\right)\phi_\lambda^{(\alpha, \beta)}(t)=0.
\end{equation}
In this paper we shall assume that $\alpha\geq \beta> -\frac 12$ and $\alpha\not=-\frac 12$. This Jacobi function can be written as hypergeometric function:

\begin{equation}
\phi_\lambda^{(\alpha, \beta)}(t)=_2F_1\left(\frac{\alpha +\beta + 1-i\lambda}{2}, \frac{\alpha + \beta + 1+ i\lambda}{2}; \alpha +1; -\sinh^2t\right).
\end{equation}

The hypergeometric function has the following integral representation for $\Re c>\Re b>0$,
\begin{eqnarray}\label{integral representation of hypergeometric function}
_2F_1(a,b;c;z)=\frac{\Gamma(c)}{\Gamma(b)\Gamma(c-b)}\int_0^1 s^{b-1}(1-s)^{c-b-1}(1-sz)^{-a}ds,\hspace{3mm}z\in\C\setminus[1, \infty).
\end{eqnarray}

Let $$L_{\alpha, \beta}=\frac{d^2}{dt^2} + \left((2\alpha + 1)\coth t + (2\beta + 1)\tanh t\right)\frac{d}{dt}.$$ Then rewriting (\ref{eqn-1}) we get that $\phi_\lambda^{(\alpha, \beta)}$ is the unique even $C^\infty$ function on $\R$ such that $\phi_\lambda^{(\alpha, \beta)}(0)=1$ and 
\begin{equation}\label{eqn-2}
(L_{\alpha, \beta} +\lambda^2 + \rho^2)\phi_\lambda^{(\alpha, \beta)}=0.
\end{equation}
Let $T^{(\alpha, \beta)}$ be the differential-difference operator defined by $$T^{(\alpha, \beta)}f(t)=f'(t) + \left((2\alpha + 1)\coth t + (2\beta + 1)\tanh t\right) \frac{f(t)-f(-t)}{2}-\rho f(-t), t\in\R.$$
The Heckman-Opdam hypergeometric functions $G_\lambda^{(\alpha, \beta)}$ on $\R$ are normalised eigenfunctions:
$$T^{(\alpha, \beta)}G_\lambda^{(\alpha, \beta)}=i\lambda G_\lambda^{(\alpha, \beta)}.$$
The functions $G_\lambda^{(\alpha, \beta)}$ are related to the Jacobi functions by 
\begin{equation}
G_\lambda^{(\alpha, \beta)}(x)=\phi_\lambda^{(\alpha, \beta)}(x) + \frac{\rho + i\lambda}{4(\alpha + 1)}\sinh 2x \phi_\lambda^{(\alpha +1 , \beta + 1)}(x).
\end{equation}
Then we have, $$\phi_\lambda^{(\alpha, \beta)}(x)=\frac{G_\lambda^{(\alpha, \beta)}(x) + G_\lambda^{(\alpha, \beta)}(-x)}{2}.$$
For $\lambda\not=-i, -2i, \cdots $, there is another solution $\Phi_\lambda^{(\alpha, \beta)}$ of (\ref{eqn-2}) on $(0, \infty)$ is given by \begin{eqnarray}\label{formula-Phi-1}
\Phi_\lambda^{(\alpha, \beta)}(t)&=& (2\cosh t)^{i\lambda-\rho}  \,_2F_1(\frac{\rho-i\lambda}{2}, \frac{\alpha-\beta +1-i\lambda}{2}; 1-i\lambda; \cosh^{-2} t) \\
&=& (2\sinh t)^{i\lambda-\rho}  \,_2F_1(\frac{\rho-i\lambda}{2}, \frac{-\alpha+\beta +1-i\lambda}{2} ; 1-i\lambda; -\sinh^{-2} t)
\label{formula-Phi-2}
\end{eqnarray}
This solution has singularity at $t=0$. For $t\rightarrow\infty$, it satisfies\begin{equation}\label{est-Phi}
\Phi_\lambda^{(\alpha, \beta)}(t)=e^{(i\lambda-\rho)t}(1 + O(1)).
\end{equation}
For $\lambda\in \C\setminus i\Z$, $\Phi_\lambda^{(\alpha, \beta)}$ and $\Phi_{-\lambda}^{(\alpha, \beta)}$ are two linearly independent solutions of (\ref{eqn-2}). So $\phi_\lambda^{(\alpha, \beta)}$ is a linear combination of both $\Phi_\lambda^{(\alpha, \beta)}$ and $\Phi_{-\lambda}^{(\alpha, \beta)}$. We have $$\phi_\lambda^{(\alpha, \beta)}=c(\lambda)\Phi_\lambda^{(\alpha, \beta)} + c(-\lambda)\Phi_{-\lambda}^{(\alpha, \beta)}$$ where $c(\lambda)$ is the Harish-Chandra $c$-function given by $$c(\lambda)=c_{(\alpha, \beta)}(\lambda)=\frac{2^{\rho-i\lambda} \Gamma(\alpha +1)\Gamma(i\lambda)}{\Gamma(\frac{\rho+ i\lambda}{2}) \Gamma(\frac{i\lambda + \alpha-\beta +1}{2})}.$$
It is normalized such that $c(-i\rho)=1$.
Hence, for $\Im\lambda<0$ and as $t\rightarrow\infty$, 
\begin{equation}\label{est-phi}
\phi_\lambda^{(\alpha, \beta)}(t)=c(\lambda)e^{(i\lambda-\rho)t}(1 + O(1)).
\end{equation}

We let $\Delta_{\alpha, \beta}(t)=(2|\sinh t|)^{2\alpha + 1} (2\cosh t)^{2\beta + 1}, t\in\R$.
The Fourier-Jacobi  transform of a suitable even function $f$ on $\R$ is defined by $$\what{f}^{(\alpha, \beta)}(\lambda)=\int_\R f(t)\phi_\lambda^{(\alpha, \beta)}(t) \Delta_{(\alpha, \beta)}(t)\, dt=2\int_0^\infty f(t)\phi_\lambda^{(\alpha, \beta)}(t) \Delta_{(\alpha, \beta)}(t)\, dt$$ for all complex numbers $\lambda$, for which the right hand side is well-defined. We point out that this definition coincides exactly with the group Fourier transform when $(\alpha, \beta)$ arises from geometric cases. 

The Fourier-Jacobi  transform of an even complex Borel measure $\mu$  is defined by $$\what{\mu}^{(\alpha, \beta)}(\lambda)=\int_\R \phi_\lambda^{(\alpha, \beta)}(t) \, d\mu(t).$$ for $\lambda\in S_1$. 

For  $f\in L^1(\R, \Delta_{\alpha, \beta})_e$, we have $\lim_{|\lambda|\rightarrow \infty}\what{f}^{(\alpha, \beta)}(\lambda)=0$ and $\lim_{|\lambda|\rightarrow \infty}\what{\mu}^{(\alpha, \beta)}(\lambda)=\mu(\{0\})$.

Also we have the following inversion formula for suitable even function $f$ on $\R$:
$$f(t)=\frac{1}{4\pi}\int_0^\infty\what{f}^{(\alpha, \beta)}(\lambda)\phi_\lambda^{(\alpha, \beta)}(t)\left|c_{(\alpha, \beta)}(\lambda)\right|^{-2}\,d\lambda.$$

The translation  of  a suitable even  function $f$ on $\R$ is given by (for all $s, t\in\R$), $$\tau_s^{(\alpha, \beta)}f(t)=\int_0^1\int_0^\pi f \left(\cosh^{-1}\left|\cosh s \cosh t +r e^{i\psi}\sinh s \sinh t\right|\right)\, dm_{\alpha, \beta}(r, \psi)$$ where the measure $dm_{\alpha, \beta}(r, \psi)$ is given by $$dm_{\alpha, \beta}(r, \psi)=\frac{2\Gamma(\alpha + 1)}{\Gamma(\frac 12)\Gamma(\alpha-\beta) \Gamma(\beta + \frac 12)} (1-r^2)^{\alpha-\beta-1} (r\sin \psi)^{2\beta} r\,dr\,d\psi$$ for $\alpha>\beta>-\frac 12$.

For $\alpha=\beta>-\frac 12$ the measure degenerates into  $$dm_{\alpha, \alpha} (r, \psi)=\frac{\Gamma(\alpha + 1)}{\Gamma(\frac 12)\Gamma(\alpha + \frac 12)}(\sin \psi)^{2\alpha}\,d\psi \, d\delta_0(r)$$ and for $\alpha>\beta=-\frac 12$ into $$dm_{\alpha, -\frac 12} (r, \psi)=\frac{2\Gamma(\alpha + 1)}{\Gamma(\frac 12)\Gamma(\alpha + \frac 12)}(1-r^2)^{\alpha-\frac 12}\,dr \frac 12 d(\delta_0 + \delta_\pi)(\psi).$$
Then it easy to check that 
\begin{enumerate}
\item $\tau_s^{(\alpha, \beta)}f(t)=\tau_t^{(\alpha, \beta)}f(s)$
\item $\tau_0^{(\alpha, \beta)}f=f$
\item $\tau_{-s}^{(\alpha, \beta)}f(t)=\tau_s^{(\alpha, \beta)}f(-t)$
\item $\tau_s^{(\alpha, \beta)}\tau_t^{(\alpha, \beta)}=\tau_t^{(\alpha, \beta)}\tau_s^{(\alpha, \beta)}$.
\item $\tau_s^{(\alpha, \beta)}\phi_\lambda^{(\alpha, \beta)}(t)=\phi_\lambda^{(\alpha, \beta)}(s)\phi_\lambda^{(\alpha, \beta)}(t)$.
\item For suitable even function $f$ on $\R$, we have $\what{\tau_s^{(\alpha, \beta)}f}(\lambda)=\phi_\lambda^{(\alpha, \beta)}(s) \what{f}^{(\alpha, \beta)}(\lambda)$.

\end{enumerate}
For suitable even functions $f$ and $g$ the convolution on $\R$ is defined by \begin{equation}\label{convolution}
f\ast_{(\alpha, \beta)} g(t)=\int_\R \left(\tau_s^{(\alpha, \beta)}f\right)(t) g(s) \Delta_{(\alpha, \beta)}(s)\,ds.
\end{equation}
Also the convolution of a sutibale even function $f$ and an even complex measure $\mu$ is defined by \begin{equation}\label{convolution-measure}
f\ast_{(\alpha, \beta)} \mu(t)=\int_\R \left(\tau_s^{(\alpha, \beta)}f\right)(t) \,d\mu(s).
\end{equation}
It is well known that $$\what{f\ast_{(\alpha, \beta)} g}(\lambda)=\what{f}(\lambda)\what{g}(\lambda)$$
and $$\|f\ast_{(\alpha, \beta)} g\|_1\leq \|f\|_1  \|g\|_1.$$

\section{The functions $b_\lambda$}
\noindent Let $\mathbb{C_+}=\{z\in\C\mid \Im z>0\}$  be the open upper half plane in $\mathbb{C}$. We fix $\alpha\geq \beta\geq -\frac 12, \alpha\not=-\frac 12$. To make expressions simplier, we shall omit indices $(\alpha, \beta)$ from $\Phi_\lambda^{(\alpha, \beta)}, \phi_\lambda^{(\alpha, \beta)},\Delta_{(\alpha, \beta)}, c_{(\alpha, \beta)}(\lambda), \cdots $etc. and write simply them as $\Phi_\lambda, \phi_\lambda, \Delta, c(\lambda), \cdots $ etc. respectively.

For $\lambda\in \mathbb{C}_+$, we define 
\begin{eqnarray}\label{definition of b-lambda}
b_\lambda(t):=\frac{i}{4\lambda c(-\lambda)}\Phi_\lambda(t), t>0
\end{eqnarray}
where $c$ is the Harish-Chandra $c$-function. We extend $b_\lambda$ evenly on $\R\setminus\{0\}$.
The function $b_\lambda$ satisfies the following properties. 
\begin{enumerate}
\item  There is a positive constant $C$ and a natural number $N$ such that for all $t\in(0,1/2]$, 
\begin{eqnarray*}
|b_\lambda(t)|\leq \begin{cases}C (1+|\lambda|)^Nt^{-2\alpha},\hspace{3mm}\textup{if}\hspace{1mm} \alpha\not=0\\
C \log \frac{1}{t}\hspace{20.5mm}\textup{if}\hspace{1mm}\alpha=0.
\end{cases}
\end{eqnarray*}
\item There is a positive constant $C$ and a natural number $M$ such that for all $t\in [1/2,\infty]$,
$$
|b_\lambda(a_t)|\leq C (1+|\lambda|)^Me^{-(\Im\lambda+\rho)t}.
$$
\item $b_\lambda$ can be written as a sum of $L^1$ and $L^p (p<2)$ functions.
\item  $b_\lambda\in L^1(\R, \Delta)_e$ if and only if $\Im\lambda>\rho$  and $\|b_\lambda\|_1\leq C\frac{(1+|\lambda|)^K}{\Im\lambda-\rho}$ for some  $K>0$.  Also, $||b_\lambda||_1\rightarrow 0$ if $\lambda\rightarrow \infty$ along the positive imaginary axis.
\item For $\lambda\in\C_+$, $\what{b_\lambda}(\xi)=\frac{1}{\xi^2-\lambda^2}, \xi\in\R$.
\item Span$\{b_\lambda\mid\Im\lambda>\rho\}$ is dense in $L^1(\R, \Delta)_e$.
\end{enumerate}
Except $(5)$, others can be proved as in \cite{PS}. So we present the proof of $(5)$ (cf. \cite[p. 128]{Dijk}).

\begin{lemma}\label{lemma-spherical transform of b-lambda}
Let $\lambda\in \mathbb{C}_+$. Then $\widehat{b}_\lambda(\xi)=\frac{1}{\xi^2-\lambda^2}$ for all $\xi\in \mathbb{R}$.
\end{lemma}

\begin{proof}
For two smooth functions $f$ and $g$ on $(0,\infty)$, we define 
$$
[f,g](t)=\Delta(t)\left[f(t)g^\prime(t)-f^\prime(t)g(t)\right],\,\,\,t>0.
$$
An easy calculation shows that $[f,g]^\prime(t)=(Lf\cdot g-f\cdot Lg)(t)\Delta(t)$. Therefore, for any $b>a>0$, we have
\begin{eqnarray}\label{e.0}
\int_a^b (Lf\cdot g-f\cdot Lg)(t)\Delta(t)=[f,g](b)-[f,g](a).
\end{eqnarray}
If $f=\phi_\lambda$ and $g=\Phi_\lambda$, then the left-hand side of the above equation is zero for all $b>a>0$, so that $[\phi_\lambda,\Phi_\lambda]$ is a (finite) constant on $(0,\infty)$, and hence
\begin{eqnarray}\label{e.01}
[\phi_\lambda,\Phi_\lambda](\cdot)=\lim_{t\rightarrow\infty}[\phi_\lambda,\Phi_\lambda]=-\lim_{t\rightarrow\infty}
\Delta(t)\left(\Phi_\lambda(t)\right)^2
\left(\frac{\phi_\lambda}{\Phi_\lambda}\right)^\prime(t)=
-\lim_{t\rightarrow\infty}
e^{2i\lambda t}\left(\frac{\phi_\lambda}{\Phi_\lambda}\right)^\prime(t)
\end{eqnarray}
by the asymptotic behaviors of $\Delta$ and $\Phi_\lambda$ at $\infty$. Again, by the asymptotic behaviors of $\phi_\lambda$ and $\Phi_\lambda$ at $\infty$, we have
$$
\lim_{t\rightarrow\infty}\frac{\frac{\phi_\lambda}{\Phi_\lambda}(t)}{e^{-2i\lambda t}}=c(-\lambda).
$$
Since, by (\ref{e.01}), 
$$
\lim_{t\rightarrow\infty}\frac{\left(\frac{\phi_\lambda}{\Phi_\lambda}\right)^\prime(t)}{-2i\lambda e^{-2i\lambda t}}
$$ 
exists, we can apply L'Hospital's rule to conclude that
$$
\lim_{t\rightarrow\infty}\frac{\left(\frac{\phi_\lambda}{\Phi_\lambda}\right)^\prime(t)}{-2i\lambda e^{-2i\lambda t}}=c(-\lambda).
$$
Hence we get $[\phi_\lambda,\Phi_\lambda](\cdot)=2i\lambda c(-\lambda)$.

Now, if $f$ is an even smooth function on $\mathbb{R}$ with $f(0)=0$, we claim that
$$
\lim_{t\rightarrow 0^+}[f,\Phi_\lambda](t)=0.
$$
To prove the claim, first note that we can assume $f$ to be compactly supported. Then with this $f$ and $g=\Phi_\lambda$, the equation \ref{e.0} (for large $b$ and $a=t\rightarrow 0^+$) implies that $\lim_{t\rightarrow 0^+}[f,\Phi_\lambda](t)$ exists.
Also we have
\begin{eqnarray*}
\lim_{t\rightarrow 0^+}[f,\Phi_\lambda](t)=-\lim_{t\rightarrow 0^+}
\Delta(t)\left(\Phi_\lambda(t)\right)^2
\left(\frac{f}{\Phi_\lambda}\right)^\prime(t)=
\begin{cases}
-\lim_{t\rightarrow 0^+}
2^{2\rho}Ct^{-2\alpha+1}\left(\frac{f}{\Phi_\lambda}\right)^\prime(t),\,\,\,\,\,\,\,\,\textup{if}\,\,\alpha\neq 0\\
-\lim_{t\rightarrow 0^+}
2^{2\rho}Ct\left(\log\frac{1}{t}\right)^2\left(\frac{f}{\Phi_\lambda}\right)^\prime(t),\,\,\,\textup{if}\,\,\alpha= 0
\end{cases}
\end{eqnarray*}
since $\Delta(t)\asymp 2^{2\rho}t^{2\alpha+1}$ and
\begin{eqnarray}
\Phi_\lambda(t)\asymp
\begin{cases} t^{-2\alpha}\,\,\,\textup{if}\,\,\alpha\neq 0\\
\log\frac{1}{t}\,\,\,\textup{if}\,\,\alpha= 0
\end{cases}
\end{eqnarray}
as $t\rightarrow 0+$. Therefore, by an application of L'Hospital's rule, the claim follows.
But, if $f(0)\neq 0$, writing
$$
[f,\Phi_\lambda]=[f-f(0)\phi_\lambda,\Phi_\lambda]+f(0)[\phi_\lambda,\Phi_\lambda],
$$
we can conclude that 
\begin{eqnarray}\label{e.1}
\lim_{t\rightarrow 0^+}[f,\Phi_\lambda](t)=2i\lambda c(-\lambda)f(0).
\end{eqnarray}

Now fix a real number $\xi$. Putting $f=\phi_\xi$ and $g=\Phi_\lambda$ in \ref{e.0}, it follows that 
$$
\int_{a}^b\Phi_\lambda(t)\phi_\xi(t)\Delta(t)dt=
\frac{1}{\lambda^2-\xi^2}\big([\phi_\xi,\Phi_\lambda](b)-
[\phi_\xi,\Phi_\lambda](a)\big).
$$
Taking limit as $a\rightarrow 0^+$, we get, by \ref{e.1},
\begin{eqnarray*}
\int_{0}^b\Phi_\lambda(t)\phi_\xi(t)\Delta(t)dt=
\frac{1}{\lambda^2-\xi^2}\big([\phi_\xi,\Phi_\lambda](b)-
2i\lambda c(-\lambda)\big).
\end{eqnarray*}
Therefore, to complete the proof it is enough to show that $[\phi_\xi,\Phi_\lambda](b)\rightarrow 0$ as $b\rightarrow\infty$. First note that the existence of (finite) limit is confirmed by the above equation itself. As in \ref{e.01}, we can write
$$
\lim_{b\rightarrow\infty}[\phi_\xi,\Phi_\lambda](b)=
\lim_{b\rightarrow\infty}
e^{2i\lambda b}\left(\frac{\phi_\xi}{\Phi_\lambda}\right)^\prime(b).
$$
By the asymptotic behavior of $\phi_\xi$ and $\Phi_\lambda$,
$$
\lim_{b\rightarrow\infty}\frac{\frac{\phi_\xi}{\Phi_\lambda}(b)}{e^{-2i\lambda b}}=0.
$$
Therefore, by L'Hospital rule,
$$
\lim_{b\rightarrow\infty}\frac{\left(\frac{\phi_\xi}{\Phi_\lambda}\right)^\prime(b)}{-2i\lambda e^{-2i\lambda b}}=0,
$$
and hence $\lim_{b\rightarrow\infty}[\phi_\xi,\Phi_\lambda](b)=0$
 as required to prove.
 \end{proof}

\section{The functions $T_\lambda f$}
Let $f\in L^1(\R, \Delta)_e$ . For each $\lambda$, with $0<\Im\lambda<\rho$, we define
\begin{eqnarray}\label{definition of T-lambda-f}
T_\lambda f:=\widehat{f}(\lambda)b_\lambda -f*b_\lambda.
\end{eqnarray}
Since $b_\lambda$ can be written as a sum of $L^1$ and $L^p$ ($p<2$) function, $T_\lambda f$ is well-defined; in fact it also has the same form i.e. can be written as a sum of $L^1$ and $L^p$ function. In particular its spherical transform is a continuous function on $\mathbb{R}$. As an easy consequence of Lemma \ref{lemma-spherical transform of b-lambda} we get, 
for $0<\Im\lambda<\rho$ and $f\in L^1(\R, \Delta)_e$, 
$$
\widehat{T_\lambda f}(\xi)=\frac{\widehat{f}(\lambda)-\widehat{f}(\xi)}{\xi^2-\lambda^2},\hspace{3mm}\textup{for all}\hspace{1mm} \xi\in\mathbb{R}.
$$

\begin{lemma}\label{lemma-functional equation of b-lambda}
Let $\lambda\in\mathbb{C}_+$. Then
\begin{eqnarray*}
\tau_sb_\lambda(t)=
\begin{cases}
b_\lambda(t)\phi_\lambda 
(s)\hspace{3mm}\textup{if}
\hspace{1mm}t> s\geq 0,\\
b_\lambda(s)\phi_\lambda(t)\hspace{3mm}\textup{if}
\hspace{1mm}s > t\geq 0.
\end{cases}
\end{eqnarray*}
\end{lemma}

\begin{proof}
First we note that if $t\neq s$, $\cosh^{-1}\left(|\cosh s\cosh t + re^{i\psi}\sinh s\sinh t|\right)$ is non zero, whatever the value of $r\in[0,1]$ and $\psi\in[0,\pi]$ be. Therefore $\tau_sb_\lambda(t)$ is well-defined whenever $t\neq s$. Since $\tau_sb_\lambda(t)=\tau_tb_\lambda(s)$, it is enough to prove the second case. Fix $s>0$. Since $b_\lambda$ is an smooth eigenfunction of $L$ on $(0,\infty)$ with eigen value $-(\lambda^2+\rho^2)$, $\tau_sb_\lambda$ is an smooth eigenfunction of $L$ on $(0,s)$ with eigenvalue $-(\lambda^2+\rho^2)$ which is regular at $0$. Therefore 
$$
\tau_sb_\lambda(t)=C\phi_\lambda(t)\hspace
{3mm}\textup{for all}\hspace{1mm} 0\leq t<s,
$$
for some constant $C$. Putting $t=0$ in the above equation we get $C=b_\lambda(s)$. Hence the proof.
\end{proof}
\noindent Using Lemma \ref{lemma-functional equation of b-lambda} $T_\lambda f, 0<\Im \lambda<\rho$ can be written as,
 $$
T_\lambda f(t)=b_\lambda(t)\int_t^\infty f(s)\phi_\lambda(s)\Delta(s)ds-\phi_\lambda(t)
\int_t^\infty f(s)b_\lambda(s)\Delta(s)ds, t>0.
$$
Using this expression of $T_\lambda f$, we can prove the following lemma (see Lemma 4.4, Remark 4.5 \cite{PS}). 
\begin{lemma}\label{L-1 norm of T-lambda-f}
Let $0<\Im\lambda<\rho$ and $f\in L^1(\R, \Delta)_e$. Also assume that $\lambda\notin B_{\rho/2}(0)$. Then 
$T_\lambda f\in L^1(\R, \Delta)_e$ and its $L^1$ norm satisfies $||T_\lambda f||_1\leq C||f||_1(1+|\lambda|)^Ld(\lambda,\partial S_1)^{-1}$, for some non-negative integer $L$, where $d(\lambda,\partial S_1)$ denotes the Euclidean distance of $\lambda$
from the boundary $\partial S_1$ of the strip $ S_1$. 
\end{lemma}

\section{Resolvent transform}
 Let $\delta$ be the Dirac delta distribution at $0$. Let $L^1_\delta(\mathbb R, \Delta)_e$ be the unital Banach algebra generated by $L^1(\mathbb R, \Delta)_e$ and $\{\delta\}$. Its maximal ideal space is one point compactification $ S_1\cup\{\infty\}$ of $ S_1$, i.e., more precisely, the maximal ideal space is $\big\{L_z:z\in S_1\cup\{\infty\}\big\}$, where $L_z$ is the complex homomorphism on $L^1_\delta(\mathbb R, \Delta)_e$ defined by $L_z(f)=\widehat{f}(z)$. For a (closed)  ideal $J$ of $L^1_\delta(\mathbb R, \Delta)_e$,  the hull $Z(J)$ is defined to be the set of common zeros (in $S_1\cup \{\infty\}$) of the Jacobi-Fourier transforms of elements in $J$. For the rest of the section $I$ always stands for a (closed) ideal of $L^1(\mathbb R, \Delta)_e$ (and hence an ideal in $L^1_\delta(\mathbb R, \Delta)_e$ too) such that the hull $Z=Z(I)$ is $\{\infty\}$ or $\{\infty,\pm i\rho\}$. Since $Z$ is the set of common zeros of Jacobi-Fourier transforms of the elements in $I$, it follows that the maximal ideal space of the quotient algebra $L^1_\delta(\mathbb R, \Delta)_e/I$ is $Z$ i.e. it consists of the complex homomorphisms $\widetilde L_z : z\in Z$, where $\widetilde{L}_z(f+I)=\hat{f}(z)$. So, by the Banach algebra theory, an element $f+I$ is invertible in $L^1_\delta(\mathbb R, \Delta)_e/I$ iff $\widehat{f}(z)\neq 0$ for all $z\in Z$.

Let $\lambda_0$ be a fixed complex number with $\Im\lambda_0>\rho$, so that  $b_{\lambda_0}$ is in $L^1$. Therefore, for $\lambda\in \mathbb{C}\setminus Z$, the function $\widehat{\delta}-(\lambda^2-
\lambda_0^2)\widehat{b}_{\lambda_0}$ does not vanish at any points of $Z$, and hence $\delta-(\lambda^2-\lambda_0^2)b_{\lambda_0}+I$ is inverible in the quotient algebra $L^1_\delta(\mathbb R, \Delta)_e/I$. We put
\begin{eqnarray}\label{definition of B-lambda}
B_\lambda=\left(\delta-(\lambda^2-
\lambda_0^2)b_{\lambda_0}+I\right)^{-1}*\left
(b_{\lambda_0}+I\right),\hspace{3mm}\lambda\in\mathbb{C}\setminus Z
\end{eqnarray}
which is, in fact, an element of $L^1(\R, \Delta)_e/I$.
Now, let $g\in L^\infty(\R, \Delta)_e$ annihilates $I$, so that we may consider $g$ as a bounded linear functional on $L^1(\R, \Delta)_e/I$. We define the resolvent tansform $\mathcal{R}[g]$ of $g$ by
\begin{eqnarray}\label{defn-R-g}
\mathcal{R}[g](\lambda)=\left\langle B_\lambda,g\right \rangle
\end{eqnarray}
From (\ref{definition of B-lambda}) it is easy to see that $\lambda\mapsto B_\lambda$ is a Banach space valued even holomorphic function on $\mathbb{C}\setminus Z$. It follows that $\mathcal{R}[g]$ is an even holomorphic function on $\mathbb{C}\setminus Z$. 

The resolvent transform $\mathcal{R}[g]$ has the following properties.
The proof of the this lemma is same as that of Lemma 5.1 in \cite{PS}. But we present the proof here since the lemma is the core of the proof  of the Wiener Tauberian theorem.
\begin{lemma} \label{properties of resolvent transform}
Assume $g\in L^\infty(\mathbb R, \Delta)_e$ annihilates $I$, and fix a function $f\in I$. Let $Z(\widehat{f}):=\{z\in S_1:\widehat{f}(z)=0\}$.

\noindent (a) $\mathcal{R}[g](\lambda)$ is an even holomorphic function on $\mathbb{C}\setminus Z$. It is given by the following formula :
\begin{eqnarray*}
\mathcal{R}[g](\lambda)=
\begin{cases} 
\langle b_\lambda,g\rangle,\hspace{3mm}\Im \lambda>\rho,\\
\frac{\langle T_\lambda f,g\rangle}{\widehat{f}(\lambda)}, \hspace{3mm}0<\Im\lambda<\rho, \lambda\notin Z(\widehat{f}).
\end{cases}
\end{eqnarray*}
\noindent (b) For $|\Im\lambda|>\rho$, $\left|\mathcal{R}[g](\lambda)\right|\leq C||g||_\infty\frac{(1+|\lambda|)^K}{d(\lambda,\partial S_1)},$

\noindent (c) For $|\Im\lambda|<\rho$, $\left|\widehat{f}(\lambda)\mathcal{R}[g](\lambda)\right|\leq C||f||_1||g||_\infty\frac{(1+|\lambda|)^L}{d(\lambda,\partial S_1)}$, where the constant $C$ is independent of $f\in I$. 
\end{lemma} 
\begin{proof}
(a) Let $\Im\lambda>\rho$. Then $b_\lambda$ is in $L^1$ and $\widehat{b}_\lambda(z)=\frac{1}{z^2-\lambda^2}$, $z\in S_1$. We observe that for $z\in S_1$,
$$
\frac{1}{\widehat{b}_{\lambda_0}(z)}-\frac{1}{\widehat{b}_\lambda(z)}=\lambda^2-\lambda_0^2
$$
which is equivalent to saying that
$$
\left(1-(\lambda^2-\lambda_0^2)
\widehat{b}_{\lambda_0}(z)\right)\widehat{b}_\lambda(z)=\widehat{b}_{\lambda_0}(z),\hspace{3mm}z\in S_1.
$$
Apply the inverse spherical transform and mod out $I$ to get
$$
\left(\delta-(\lambda^2-\lambda_0^2)
b_{\lambda_0}+I\right)*(b_\lambda+I) =b_{\lambda_0}+I,
$$
Since $\left(\delta-(\lambda^2-\lambda_0^2)
b_{\lambda_0}+I\right)$ is invertible in $L^1_\delta(\mathbb R)_e/I$, comparing the above equation with \ref{definition of B-lambda} we get $B_\lambda=b_\lambda+I$. Therefore, by the definition of $\mathcal{R}[g](\lambda)$, $\mathcal{R}[g](\lambda)=\langle b_\lambda,g\rangle$.

Next we assume that $0<\Im\lambda<\rho$, $\lambda\notin Z(\widehat{f})$. So, $T_\lambda f$ is in $L^1$ and $\widehat{T_\lambda f}(z)=\frac{\widehat{f}(\lambda)-\widehat{f}(z)}{z^2-\lambda^2}, z\in S_1$. A small calculation shows that
$$
\left(1-(\lambda^2-\lambda_0^2)
\widehat{b}_{\lambda_0}(z)\right)\frac{\widehat{T_\lambda f}(z)}{\widehat{f}(\lambda)}=\widehat{b}_{\lambda_0}(z)-\frac{\widehat{f}(z)\widehat{b}_{\lambda_0}(z)}{\widehat{f}(\lambda)},\hspace{3mm}z\in S_1.
$$
Again, apply inverse spherical transform and mod out $I$ to get
$$
\left(\delta-(\lambda^2-\lambda_0^2)
b_{\lambda_0}+I\right)*\left(\frac{T_\lambda f}{\widehat{f}(\lambda)}+I\right) =b_{\lambda_0}+I.
$$
Therefore $B_\lambda=\frac{T_\lambda f}{\widehat{f}(\lambda)}+I$ which gives the desired formula for $\mathcal{R}[g](\lambda)$ in this case.
 
(b) It follows from the estimate of $\|b_\lambda\|_1$ and the fact that $\mathcal{R}[g](\lambda)$ is even.

(c) From Lemma \ref{L-1 norm of T-lambda-f} it follows that $$\left|\widehat{f}(\lambda)\mathcal{R}[g](\lambda)\right|\leq C||f||_1||g||_\infty\frac{(1+|\lambda|)^L}{d(\lambda,\partial S_1)}$$ for $0<\Im\lambda<\rho/2,\lambda\not\in B_{\rho/2}(0)$, where $C$ is independent of $f\in I$. Since $\widehat{f}(\lambda)\mathcal{R}[g](\lambda)$ is an even continuous function on $ S_1$, the same estimate is true for $|\Im\lambda|<\rho, \lambda\not\in B_{\rho/2}(0)$. From (\ref{defn-R-g}) it follows that  $\mathcal{R}[g](\lambda)$ is bounded on $ B_{\rho/2}(0)$, with bound independent of $f$. Therefore  on $ B_{\rho/2}(0)$$$\left|\widehat{f}(\lambda)\mathcal{R}[g](\lambda)\right|\leq C||f||_1$$ where $C$ is independent of $f$. Hence the proof follows.
\end{proof}

\section{some results from complex analysis}
In this section we state some results from complex analysis. The proof of them involves the log-log theorem, the Paley-Wiener theorem, Alhfors distortion theorem, and the Phragman-Lindel\''{o}f principle (\cite{Hedenmalm-1985}, \cite{Dah}).

For any function $F$ on $\R$, we let $$\delta_\infty^+(F)=-\limsup_{t\rightarrow\infty} e^{-\frac{\pi}{2\rho}t}\log|F(t)|, \,\,\,\,\,\,\delta_\infty^-(F)=-\limsup_{t\rightarrow\infty} e^{-\frac{\pi}{2\rho}t}\log|F(-t)|$$ and $$\delta_{ i\rho}(F)=\limsup_{t\rightarrow \rho-} (\rho-t)\log|F(i t)|, \,\,\,\,\,\,\delta_{-i\rho}(F)=\limsup_{t\rightarrow (-\rho)+} (\rho + t)\log|F(i t)|.$$

Proof of the following theorem is similar to \cite[Theorem 6.3]{PS}.
\begin{theorem}\label{loglog theorem}
Let $\Omega$ be a collection of bounded holomorphic functions $F$ on $S_1^0$ such that 
$$
\inf_{F\in \Omega}\delta^+_\infty(F)=\inf_{F\in \Omega}\delta^-_\infty(F)=0.
$$ Suppose $H$ is a holomorphic function on $\mathbb{C}\setminus \{\pm i\rho\}$ such that, for some non-negative integer $N$, it satisfies the following estimates :
\begin{eqnarray*}
|H(z)|&\leq &\frac{(1+|z|)^N}{d(z,\partial S_1)},
\hspace{3mm}z\in\mathbb{C}\setminus S_1,\\
|F(z)H(z)|&\leq & \frac{(1+|z|)^N}{d(z,\partial S_1)},
\hspace{3mm}z\in S_1^0,
\hspace{1mm}\textup{for all}\hspace{1mm}F\in \Omega.
\end{eqnarray*}
Then $H$ is dominated by a polynomial outside a bounded neighbourhood of $\{\pm i\rho\}$. 
\end{theorem}

The following theorem follows from the proof of  \cite[Theorem 6.13]{Dah}:
\begin{theorem}
Let $\Omega$ be a collection of bounded holomorphic functions $F$ on $ S_1^0$ such that $|F(z)|\rightarrow 0$ as $|z|\rightarrow \infty$ (in $S_1^\circ$) and 
$$
\inf_{F\in \Omega}\delta_{\pm i\rho}(F)=0.$$
Suppose $G$ is a holomorphic function on $\mathbb{C}\setminus Z$ ($Z$ is a finite subset of $\partial S_1$) such that for some positive integer $N$ it satisfies the following estimate :
\begin{eqnarray*}
|F(z)G(z)|&\leq & (d(z,\partial S_1))^{-N},\hspace{3mm}z\in S_1^0,
\hspace{1mm}\textup{for all}\hspace{1mm}F\in \Omega.
\end{eqnarray*} 
Then $G$ has  poles at $\pm i\rho$ of order atmost $N$.
\end{theorem}

\begin{theorem}\label{pole}
Let $\Omega$ be a collection of bounded holomorphic functions $F$ on $ S_1^0$ such that $|F(z)|\rightarrow 0$ as $|z|\rightarrow \infty$ (in $S_1^\circ$) and 
$$\inf_{F\in \Omega}\delta_{\pm i\rho}(F)=0.$$
Suppose $H$ is a holomorphic function on $\mathbb{C}\setminus \{\pm i\rho\}$  satisfying the following estimate (for some positive integer $N$) :
\begin{eqnarray*}
|F(z)H(z)|&\leq &  \frac{(1+|z|)^{N}}{d(z,\partial S_1)},\hspace{3mm}z\in S_1^0,
\hspace{1mm}\textup{for all}\hspace{1mm}F\in \Omega.
\end{eqnarray*} 
Then $G$ has at most simple poles at $\pm i\rho$.
\end{theorem}

\begin{proof}
We can assume that $N$ is even. We define the holomorphic function $G$ on $\mathbb{C}\setminus\{\pm i\rho\}$ by
$$G(z)=\frac{H(z)}{(z-i\rho)^{N/2}(z+i\rho)^{N/2}}.$$
Then clearly $G(z)$ satisfies
\begin{eqnarray*}
|F(z)G(z)|&\leq & \frac{1}{(d(z,\partial S_1))^{N/2+1}},
\hspace{3mm}z\in S_1^0,
\hspace{1mm}\textup{for all}\hspace{1mm}F\in \Omega.
\end{eqnarray*}
Hence the theorem follows by the previous theorem. 
\end{proof}

\section{Proof of the main theorem}
\begin{proof}[proof of Theorem \ref{main-theorem}:]
Proof of \textbf{(1)} is similar to ``{\em proof of Theorem 1.2}'' in Section $7$ (in \cite{PS}).

\textbf{(2)} We can assume that the elements in $I$ are of unit norm. Let $g\in L^\infty(\R, \Delta)_e$ annihilates the (closed) ideal $I$ generated by $\{f_\nu\mid \nu\in\Lambda\}$. We must show that $g$ annihilates $L_0^1(\R, \Delta)_e$. By Lemma \ref{properties of resolvent transform}, $\mathcal{R}[g]$ satisfies the following estimates
\begin{eqnarray*}
|\mathcal{R}[g](z)|&\leq & C(1+|z|)^N\left(d(z,\partial S_1)\right)^{-1},
\hspace{3mm}z\in\mathbb{C}\setminus S_1,\\
|\widehat{f_\nu}(z)\mathcal{R}[g](z)|&\leq & C(1+|z|)^N\left(d(z,\partial S_1)\right)^{-1},
\hspace{3mm}z\in S_1^0,
\end{eqnarray*}
for all $\nu\in\Lambda$, for some constant $C$. 
Therefore, by Theorem \ref{pole}, it has at most simple poles at $\{\pm i\rho\}$. So we  write 
$$
\mathcal{R}[g](z)=\frac{a}{z^2+\rho^2}+h(z),\,\,\,
z\in\mathbb{C}\setminus\{\pm i\rho\}
$$ 
for some constant $a$ and even entire function $h$. Also, by Theorem \ref{loglog theorem} $\mathcal{R}[g]$ has at most polynomial growth at $\infty$, and by $(4)$ (in section $3$),
$\mathcal{R}[g](z)\rightarrow 0$ as $|z|\rightarrow\infty$ along the imaginary axis. Therefore the same properties are satisfied by the function $h$ too, so that by Liouville's theorem $h=0$, and hence 
$$
\mathcal{R}[g](z)=\frac{a}{z^2+\rho^2},\,\,\, z\in\mathbb{C}\setminus\{\pm i\rho\}.
$$
Let $m\in L^\infty(\R, \Delta)_e$ corresponds to the complex homomorphism $f\mapsto\hat{f}(i\rho)$ on $L^1(\R, \Delta)_e$ i.e. $\hat{f}(i\rho)=\langle f,m\rangle$ for all $f\in L^1(\R, \Delta)_e$. Then for $z$ with $\Im z>0$,
$$
\mathcal{R}[g](z)=-a \hat{b}_{z}(i\rho)=-a\langle{b_z, m\rangle}.
$$ 
Since $\{b_z:\Im z>0\}$ is dense in $L^1(\R, \Delta)_e$, $g=-\bar am$. Since $m$ annihilates $L^1_0(\R, \Delta)_e$, so does $g$.
\end{proof}

\section{Furstenberg Theorem}
Let $G$ be a noncompact connected semisimple Lie group with finite center and $K$ be a maximal compact subgroup of $G$. Let $\mu$ be a $K$-invariant complex measure on $G/K$ such that $\mu(G/K)=1$. A bounded function $f$ on $G/K$ is said to be $\mu$-harmonic if $f*\mu=f$ i.e.
$$
\int_Gf(gh)d\mu(h)=f(g),\,\,\,\,\textup{for all}\,\,g\in G.
$$
If $f$ is harmonic (i.e. $\int_K f(gkh)dk=f(g)$ for all $g, h\in G$, or, equivalenttly, $f$ is annihilated by the Laplace-Beltrami operator), it is easy to see that it is $\mu$-harmonic. Naturally the following question arises :

\textbf{(A)} Under what conditions on $\mu$, $\mu$-harmonic functions are  hamonic functions only?
 
In \cite[Theorem 5, p. 370]{Furs} Furstenberg answers the question above  in positive, when $\mu$ is absolutely continuous $K$-invariant probability measure on $G/K$. In \cite{Ben-2} using Winner-Tauberian Theorem, the authors proved the following result for the disc algebra $\mathbb D=\mathrm{SL_2}(\mathbb{R})/\mathrm{SO}(2)$.  Let $\Sigma$ denote the usual maximal ideal space $\{z\in\mathbb C:0\leq\Re z\leq 1\}$.

\begin{theorem}\label{theorem-furstenberg-D}
Let $\mu$ be a $SO(2)$-invariant complex measure on $\mathbb D$ such that $\mu(\mathbb D)=1$, $\mu(\{0\})\neq 1$, $\widehat{\mu}(\lambda)\neq 1$ for all $\lambda\in\Sigma\setminus\{0,1\}$, and 
$$
\limsup_{x\rightarrow 0^+} x\log|1-\widehat{\mu}(x)|=0.
$$
Then every $\mu$-harmonic functions are essentially the harmonic ones.
\end{theorem}

Note that the theorem above includes the complex measure too unlike the Furstenberg theorem where the measure is essentially positive. They have also proved that any probabilty measure $\mu$ with $\mu(\{0\})\neq 1$ satisfies all the conditions of the above theorem. Hence for $\mathrm{SL}(2, \R)$ their result contains the Furstenberg Theorem as a particular case. 

Since, in this paper, we have obtained the similar Winner-Tauberian Theorem for general hypergeometric transforms  (which include all real rank one cases), it is natural to expect that the theorem above holds true for hypergeometric cases. The notion of `$\mu$-harmonic' does not make sense in general unless the pair $(\alpha,\beta)$ arises from a geometric case. But this difficulty can be overcome by the following observation. If $G$ is of rank one symmetric space, then writing the Cartan decomposition $G=KAK$, we can identify $K$-biinvariant functions on $G$ with even functions on $A=\mathbb{R}$. Therefore, taking an average over $K$, we can write the problem \textbf{(A)} in the following equivalent form ((see the proof of \cite[Theorem 3.1]{Ben-0})):

\textbf{(B)} Let $\mu$ be an even complex measure on $\mathbb R$ such that $\mu(\R)=1$. Then under what condition on $\mu$, the only even bounded solutions (on $\mathbb R$) of the equation $f\ast\mu=f$ are the constant functions. Here,  the convolution $\ast$ is defined by (\ref{convolution-measure}).

Now we are in position to state the analogues of Theorem \ref{theorem-furstenberg-D} for our hypergeometric cases. Before stating the theorem, we point out that the choice of maximal ideal space is horizontal strip in our case, where as their is vertical.

\begin{theorem}\label{theorem-furstenberg}
Let $\mu$ be an even complex measure on $\mathbb{R}$ such that $\mu(\R)=1$, $ \mu(\{0\})\neq 1$, $\widehat{\mu}(\lambda)\neq 1$ for all $\lambda\in S_1\setminus \{\pm i\rho\}$, and 
\begin{eqnarray}\label{a}
\limsup_{x\rightarrow \rho-}(\rho-x)\log|1-\widehat{\mu}(ix)|=0.
\end{eqnarray}
Then the only even bounded solutions of the equation $f*\mu=f$ are the constant functions. 
\end{theorem}

\begin{proof}
If $f$ is constant function then, it is easy to see that, 
$\tau_sf(t)=f(t)$ for all $s,t\in\R$. Therefore 
$$
f*\mu(t)=f(t)\int_\R d\mu(s)=f(t),\,\, t\in\R.
$$ 

Conversely, let $f$ be an even bounded function on $\mathbb R$ such that $f*\mu=f$. We need to show that $f$ is a constant function. The proof is essentially same as that of the Theorem \ref{theorem-furstenberg-D}. Let $I$ be the closed ideal in $L^1_{0}(\mathbb{R}, \Delta)_e$ generated by $\mathfrak{S} =(\mu-\delta)*L^1(\mathbb{R}, \Delta)_e$. We shall show that $\mathfrak{S}$ satisfies all the conditions of Theorem \ref{main-theorem}(2). Since $\widehat{\mu}(\lambda)\neq 1$ for all $\lambda \in S_1\setminus \{\pm i\rho\}$, the common zero set of Foureir-Jacobi transforms of the elements in $\mathfrak{S}$ is $\{\pm i\rho\}$. 
Also we have, $\widehat{\mu}(t)\rightarrow 
\mu(\{0\})$ as $t\rightarrow\infty$. But its 
given that $\mu(\{0\})\neq 1$. Therefore it 
follows that $\mathfrak{S}$ contains an $g$ 
such that $\delta_\infty^+(g)=0$. Also (\ref{a}) implies that $\mathfrak{S}$ contains an element $h$ such that $\delta_{i\rho}(h)=0$. Hence, by Theorem \ref{main-theorem} (2), we can conclude that $I=L_0^1(\mathbb R, \Delta)_e$. Since $f*\mu=f$, clearly, $f*\mathfrak{S}=0$, and hence $f*L^1_0(\mathbb R, \Delta)_e=0$ which implies that $f$ is a constant function. 
\end{proof}

\begin{corollary}\label{corollary-furstenberg}
Let $\mu$ be an even probability measure such that  $\mu(\{0\})\neq 1$. Then the only even bounded solutions of the equation $f*\mu=f$ are the constant functions. 
\end{corollary}

Again the proof of the above corollary is same as that of  \cite[Corollary 7.2]{Ben-2}, once we have the following two lemmas. 
We shall make use the following derivation property of the hypergeometric function (see \cite[p.241, eqn. (9.2.2)]{Lebedev}):
$$\frac{d}{dz} {}_2F_1(a,b;c;z)=\frac{ab}{c}  {}_2F_1(a+1,b+1;c+1;z), z\in \C\setminus [1, \infty].$$

Helgason-Johnson's theorem states that $|\phi_\lambda|\leq 1$ if and only if $\lambda\in S_1$. We have the following:
\begin{lemma}\label{phi strictly less than one}
$|\phi_\lambda(t)|<1$ for all $t>0$ if and only if $\lambda\in S_1\setminus\{\pm i\rho\}$.
\end{lemma} 
\begin{proof}
Case 1 : $\lambda\in S_1^0$. Then for $t>0$ 
\begin{eqnarray*}
|\phi_\lambda(t)|\leq \phi_{i\Im \lambda}(t) &=&_2F_1\left(\frac{\rho+\Im\lambda}{2},\frac{\rho-\Im\lambda}{2};\alpha+1;-\sinh^2t\right)\\
&=&\frac{\Gamma(\alpha+1)}{\Gamma\left(\frac{\rho+\Im\lambda}{2}\right)\Gamma\left(\frac{\rho-\Im\lambda}{2}\right)}\int_0^1s^{\frac{\rho-\Im\lambda}{2}-1}(1-s)^{\frac{\alpha-\beta+1+\Im\lambda}{2}-1}(1+s\sinh^2t)^{-\frac{\rho+\Im\lambda}{2}}ds\\
&<&\frac{\Gamma(\alpha+1)}{\Gamma\left(\frac{\rho+\Im\lambda}{2}\right)\Gamma\left(\frac{\rho-\Im\lambda}{2}\right)}\int_0^1s^{\frac{\rho-\Im\lambda}{2}-1}(1-s)^{\frac{\alpha-\beta+1+\Im\lambda}{2}-1}ds\\
&=&\phi_{i\Im\lambda}(0)=1.
\end{eqnarray*}

Case 2 : $\lambda=a+i\rho$, $a\neq 0$. Recall the function $G_\lambda^{(\alpha,\beta)}$ from preliminaries.
$$
G_\lambda^{(\alpha,\beta)}(t)=\phi_\lambda^{(\alpha,\beta)}(t)+\frac{\rho+i\lambda}{4(\alpha+1)}(\sinh 2t)\phi_\lambda^{(\alpha+1,\beta+1)}(t),\,\,\,t\in\mathbb R.
$$
Using the derivation formula of hypergeometric function, we can evaluate $\frac{d^2}{dt^2}\left|G^{(\alpha,\beta)}_{a+i\rho}\right|^2(0)$ to be equal to $-\frac{a^2(2\alpha+1)}{2(\alpha+1)^2}$ which is non-zero. Therefore, $\left|G_{a+i\rho}^{(\alpha,\beta)}\right|^2$ being analytic it can not be identically $1$. By \cite[Proposition 3.1]{Schapira}, $|G^{(\alpha,\beta)}_{a+i\rho}(t)|\leq G_{i\rho}(t)=1$ for all $t\in\mathbb R$. Since $|G^{(\alpha,\beta)}_{a+i\rho}(0)|^2=1$, we can have $\epsilon>0$ such that $|G^{(\alpha,\beta)}_{a+i\rho}(t)|^2<1$ for all non-zero $t$ with $|t|\leq \epsilon$. But, in the proof of the \cite[Proposition 3.1]{Schapira} it is shown that
$$
t\rightarrow \max\left\{{|G^{(\alpha,\beta)}_{a+i\rho}(t)|^2,|G^{(\alpha,\beta)}_{a+i\rho}(-t)|^2}\right\}
$$
is a decreasing function of $t\geq 0$.  So it follows that $|G^{(\alpha,\beta)}_{a+i\rho}(t)|$ is strictly less than $1$ for all $t\neq 0$. Since
$$
\phi_\lambda^{(\alpha,\beta)}(t)=\frac{1}{2}\left[G_\lambda^{(\alpha,\beta)}(t)+G_\lambda^{(\alpha,\beta)}(-t)\right],\,\,\ t\in\mathbb R,
$$
the proof follows. 

\end{proof}

\begin{lemma}\label{lemma-derivative of phi}
Let $t>0$. Then $\frac{d}{dx}\mid_{x=\rho}[\phi_{ix}(t)]>0$.
\end{lemma}
\begin{proof}
Define the function $g$ on $\mathbb{R}$ by
$$
g(t)=\frac{d}{dx}\bigg\vert_{x=\rho}[\phi_{ix}(t)].
$$
Then $g(0)=0$ and
\begin{eqnarray*}
g^\prime (t)&=&\frac{d}{dx}\bigg\vert_{x=\rho}\left(\frac{d}{dt}[\phi_{ix}(t)]\right)\\
&=&\frac{d}{dx}\bigg\vert_{x=\rho}\left[\frac{(\frac{\rho+x}{2})(\frac{\rho-x}{2})}{\alpha+1}\,_2F_1\left(\frac{\rho+x}{2}+1,\frac{\rho-x}{2}+1;\alpha+2;-\sinh^2 t)\right)(-\sinh 2t)\right]\\
&=& \frac{\rho\sinh t}{2(\alpha+1)}\,_2F_1\left(\rho+1,1;\alpha+2;-\sinh^2t\right)\\
&=&
\frac{\sinh t}{2\Gamma(\rho)}\int_0^\infty (1-s)^\alpha(1+s\sinh^2t)^{-\rho-1}ds
\end{eqnarray*}
which is strictly positive whenever $t>0$. Therefore $g$ is strictly increasing function on $[0,\infty)$ and hence $g(t)>0$ for all $t>0$.
\end{proof}
\begin{proof}
\textit{of Corollary \ref{corollary-furstenberg}} : We only need to show that the measure $\mu$ satisfies all the conditions of Theorem \ref{theorem-furstenberg}. Since $\mu$ is a probability measure $\mu(\mathbb R)=1$; $\mu(\{0\})\neq 1$ is given, in fact,  $\mu(\{0\})< 1$ since $\mu$ is positive. If $\lambda\in S_1\setminus\{\pm i\rho\}$, then by Lemma \ref{phi strictly less than one},
$$
|\widehat{\mu}(\lambda)|\leq \int_\R|\phi_\lambda^{(\alpha,\beta)}(t)|d\mu(t)<\int_\R d\mu(t)=1.
$$
Therefore we only left to show that $
\limsup_{x\rightarrow \rho-}(\rho-x)\log|1-\widehat{\mu}(ix)|=0.
$ 
For $t\geq 0$, let 
$$
L(t)=\frac{d}{dx}\mid_{x=\rho}[\phi_{ix}^{(\alpha,\beta)}(t)].
$$ Since, by Lemma \ref{lemma-derivative of phi}, $L(t)$ is strictly positive for all $t>0$ and, by the given condition, $\mu$ is is not concentrated at $0$ there exist $b>a>0$ such that 
$$
\int_a^b L(t)d\mu(t)>0.
$$ 
Fix $0<\epsilon<\rho$. From the Taylor series expansion (upto second order) of the function $x\rightarrow \phi_{ix}(t)$ at the point $x=\rho$, it follows that
$$
1-\phi_{ix}(t)\asymp(\rho-x)L(t),\,\,\textup{for all}\,\,x\in[\rho-\epsilon,\rho], t\in[a,b].
$$
Therefore, for all $x\in[\rho-\epsilon,\rho]$,
$$
1-\widehat{\mu}(ix)=2\int_0^\infty\left(1-\phi_{ix}(t)\right)d\mu(t)\geq 2\int_a^b\left(1-\phi_{ix}(t)\right)d\mu(t)\geq 2C(\rho-x)\int_a^bL(t)d\mu(t)
$$
where $C$ is a positive constant which depends only on $\epsilon, a, b$. Therefore it follows that 
$$
\limsup_{x\rightarrow \rho-}(\rho-x)\log|1-\widehat{\mu}(ix)|=0.
$$
as desired. 
\end{proof}


\begin{thebibliography}{99}


\bibitem{Ben-0} Benyamini, Y.; Weit, Y. {\em Harmonic analysis of spherical functions on SU(1,1).}  Ann. Inst. Fourier (Grenoble) 42 (1992), no. 3, 671--694. 

\bibitem{Ben-1}  Ben Natan, Y.; Benyamini, Y.; Hedenmalm, H.; Weit, Y. {\em Wiener's Tauberian theorem in $L^1(G//K)$ and harmonic functions in the unit disk.}  Bull. Amer. Math. Soc. (N.S.) 32 (1995), no. 1, 43--49.

\bibitem{Ben-2}  Ben Natan, Y.; Benyamini, Y.; Hedenmalm, H.; Weit, Y. {\em Wiener's Tauberian theorem for spherical functions on the automorphism group of the unit disk.} Ark. Mat. 34 (1996), no. 2, 199--224.

%\bibitem{Coddington}  Coddington, E. A.; Levinson, N.  {\em Theory of ordinary differential equations.} McGraw-Hill Book Company, Inc., New York-Toronto-London, 1955.

\bibitem{Dah}
 Dahlner, ~A. {\em A Wiener Tauberian theorem for weighted convolution algebras of zonal functions on the automorphism group of the unit disc}, 67--102, Contemp. Math., 404, Amer. Math. Soc., Providence, RI, 2006. 
 
 %\bibitem{Erdelyi-1} Erd\'{e}lyi, A.; Magnus, W.; Oberhettinger, F.; Tricomi, F. G. {\em Higher transcendental functions. Vol. I.}  McGraw-Hill Book Company, Inc., New York-Toronto-London, 1953..
 
 
\bibitem{EM1}
Ehrenpreis, L.; Mautner, F. I. {\em Some properties of the Fourier transform on semisimple Lie groups}. I.  Ann. of Math. (2)  61,  (1955). 406--439. 

%\bibitem{EM2} Ehrenpreis, ~L.; Mautner, ~F. ~I. {\em Uniformly bounded representations of groups}.Proc. Nat. Acad. Sci. U. S. A.  41  (1955), 231--233. 

%\bibitem{FJ}Flensted-Jensen, ~M. {\em Paley-Wiener type theorems for a differential operator connected with symmetric spaces}.  Ark. Mat.  10 143--162. (1972). 


\bibitem{Furs}Furstenberg, H. {\em Boundaries of Riemannian symmetric spaces.} Symmetric spaces (Short Courses, Washington Univ., St. Louis, Mo., 1969–1970), pp. 359–377. Pure and Appl. Math., Vol. 8, Dekker, New York, 1972. 

%\bibitem{Grad}  Gradshteyn, I. S.; Ryzhik, I. M. {\em Table of integrals, series, and products.}  Seventh edition. Elsevier/Academic Press, Amsterdam, 2007. 


\bibitem{GV} Gangolli, ~R; Varadarajan, V. S. {\em Harmonic analysis of spherical functions on real reductive groups}. Ergebnisse der Mathematik und ihrer Grenzgebiete, 101. Springer-Verlag, Berlin, 1988.   


\bibitem{Hedenmalm-1985}  Hedenmalm, H. {\em On the primary ideal structure at infinity for analytic Beurling algebras.}  Ark. Mat. 23 (1985), no. 1, 129--158.
\bibitem{Helga-GGA} Helgason, ~S. {\em Groups and geometric analysis. Integral geometry, invariant differential operators, and spherical functions}.  Academic Press, Inc., Orlando, FL, 1984. MR0754767 (86c:22017)


\bibitem{Koornwinder} Koornwinder, T. H. {\em Jacobi functions and analysis on noncompact semisimple Lie groups.} Special functions: group theoretical aspects and applications, 1--85, Math. Appl., Reidel, Dordrecht, 1984.

%\bibitem{K-S} Kunze, R. A.; Stein, E. M. {\em Uniformly bounded representations and harmonic analysis of the $2\times 2$ real unimodular group}. Amer. J. Math. 82 1960 1--62. 

\bibitem{Lebedev}  Lebedev, N. N. {\em Special functions and their applications. Revised English edition.} Translated and edited by Richard A. Silverman Prentice-Hall, Inc., Englewood Cliffs, N.J. 1965.

\bibitem{Liu}  Liu, Jianming; Zheng, Weixing {\em A Wiener-Tauberian theorem for Fourier-Jacobi transform.} Approx. Theory Appl. (N.S.) 13 (1997), no. 2, 97--104.

\bibitem{Naru-2009} Narayanan, E. K. {\em Wiener Tauberian theorems for $L^1(K\backslash G/K)$.}  Pacific J. Math. 241 (2009), no. 1, 117--126.

 \bibitem{Naru-2011} Narayanan, E. K.; Sitaram, A. {\em Analogues of the Wiener Tauberian and Schwartz theorems for radial functions on symmetric spaces.} Pacific J. Math. 249 (2011), no. 1, 199--210. 

\bibitem{Pusti-2011} Pusti, S.; Ray, S. K.; Sarkar, R. P. {\em Wiener-Tauberian type theorems for radial sections of homogeneous vector bundles on certain rank one Riemannian symmetric spaces of noncompact type.} Math. Z. 269 (2011), no. 1-2, 555--586.

\bibitem{PS} Pusti, S.; Samanta, A. {\em Wiener-Tauberian theorem for rank one semisimple Lie groups. } arXiv:1508.06199

\bibitem{Sarkar-1997} Sarkar, R. P. {\em Wiener Tauberian theorems for $\mathrm{SL}(2, \R)$}. Pacific J. Math. 177 (1997), no. 2, 291--304.

\bibitem{Sarkar-1998} Sarkar, R. P. {\em Wiener Tauberian theorem for rank one symmetric spaces.} Pacific J. Math. 186 (1998), no. 2, 349--358. 

\bibitem{Schapira}  Schapira, B. {\em Contributions to the hypergeometric function theory of Heckman and Opdam: sharp estimates, Schwartz space, heat kernel.} Geom. Funct. Anal. 18 (2008), no. 1, 222–250.

\bibitem{Sitaram-1980}  Sitaram, A. {\em  An analogue of the Wiener-Tauberian theorem for spherical transforms on semisimple Lie groups.} Pacific J. Math. 89 (1980), no. 2, 439--445.

\bibitem{Sitaram-1988} Sitaram, A. {\em On an analogue of the Wiener Tauberian theorem for symmetric spaces of the noncompact type.} Pacific J. Math. 133 (1988), no. 1, 197--208. 

\bibitem{Dijk} van Dijk, Gerrit {\em Introduction to harmonic analysis and generalized Gelfand pairs.} de Gruyter Studies in Mathematics,  ISBN: 978-3-11-022019-3 

\end{thebibliography}
\end{document}